\documentclass[12pt,reqno]{amsart}
 
\usepackage[euler-digits]{eulervm}
\usepackage{graphicx}
\usepackage{amsfonts, amsthm, amssymb, amsmath,mathtools}
\usepackage{stmaryrd}
\usepackage{mathrsfs,array}
\usepackage{eucal,times,color,enumerate,accents}
\usepackage{url}
\usepackage{float}
\usepackage{pbox}
\usepackage[headings]{fullpage}

\usepackage{ytableau}
\usepackage{color}
\usepackage{mathrsfs}
\usepackage{amssymb}
\usepackage{bm}
\usepackage{enumerate}
\usepackage{amssymb}
\usepackage{hyperref, cleveref}
\usepackage[all,cmtip]{xy}
\usepackage{comment}

\newcommand{\ncom}{\newcommand}

\ncom{\dho}{\partial}
\ncom{\rar}{\rightarrow}
\ncom{\imply}{\Rightarrow}
\ncom{\lrar}{\longrightarrow}
\ncom{\into}{\hookrightarrow}
\ncom{\onto}{\twoheadrightarrow}
\ncom{\ov}{\overline}
\ncom{\m}{\mbox}
\ncom{\sta}{\stackrel}
\ncom{\invlim}{\varprojlim}
\ncom{\xhat}{\widehat}

\ncom{\vspc}{\vspace{3mm}}
\ncom{\End}{{\cE}nd}
\ncom{\tensor}{\otimes}

\ncom{\al}{\alpha}
\ncom{\cHom}{{\mathcal Hom}}

\ncom{\A}{{\mathbb A}}
\ncom{\comx}{{\mathbb C}}
\ncom{\E}{{\mathbb E}}
\ncom{\F}{{\mathbb F}}
\ncom{\G}{{\mathbb G}}
\ncom{\K}{{\mathbb K}}
\ncom{\Le}{{\mathbb L}}
\ncom{\N}{{\mathbb N}}

\ncom{\p}{{\mathbb P}}
\ncom{\Q}{{\mathbb Q}}
\ncom{\R}{{\mathbb R}}
\ncom{\Z}{{\mathbb Z}}

\ncom{\f}{\dfrac}

\ncom{\wtil}{\widetilde}

\ncom{\ci}{{\mathpzc i}}

\ncom{\cA}{{\mathcal A}}
\ncom{\cC}{{\mathcal C}}
\ncom{\cE}{{\mathcal E}}
\ncom{\cF}{{\mathcal F}}
\ncom{\cG}{{\mathcal G}}
\ncom{\cH}{{\mathcal H}}
\ncom{\cI}{{\mathcal I}}
\ncom{\cJ}{{\mathcal J}}
\ncom{\cK}{{\mathcal K}}
\ncom{\cL}{{\mathcal L}}
\ncom{\cM}{{\mathcal M}}
\ncom{\cN}{{\mathcal N}}
\ncom{\cO}{{\mathcal O}}
\ncom{\cP}{{\mathcal P}}
\ncom{\cQ}{{\mathcal Q}}
\ncom{\cR}{{\mathcal R}}
\ncom{\cS}{{\mathcal S}}
\ncom{\cT}{{\mathcal T}}
\ncom{\cU}{{\mathcal U}}
\ncom{\cV}{{\mathcal V}}
\ncom{\cW}{{\mathcal W}}
\ncom{\cX}{{\mathcal X}}
\ncom{\cY}{{\mathcal Y}}
\ncom{\cZ}{{\mathcal Z}}

\ncom{\cSU}{{\mathcal S \mathcal U}}
\ncom{\eop}{{\hfill $\Box$}}
\ncom{\isom}{\cong}

\newcommand\scalemath[2]{\scalebox{#1}{\mbox{\ensuremath{\displaystyle #2}}}}

\theoremstyle{plain}
\newtheorem{theorem}{Theorem}
\newtheorem{lemma}[theorem]{Lemma}

\newtheorem{conj}{Conjecture}

\newtheorem{corollary}[theorem]{Corollary}
\newtheorem{proposition}[theorem]{Proposition}
\theoremstyle{definition}

\newtheorem{defn}{Definition}

\theoremstyle{remark}
\newtheorem{remark}{Remark}

\newtheorem{eg}{Example}

\long\def\comment#1{}

\begin{document}
	
	\title{Some remarks on two-periodic modules over local rings}
	\author[N. Das]{Nilkantha Das}
	
	\address{Stat-Math Unit, Indian Statistical Institute, 203 B.T. Road, Kolkata 700 108, India.}
	\email{dasnilkantha17@gmail.com}
	\author[S. Dey]{Sutapa Dey}
	\address{Department of Mathematics, Indian Institute of Technology -- Hyderabad, 502285, India.}
	\email{ma20resch11002@iith.ac.in}
	
	
	\begin{abstract}
In this note, some properties of finitely generated two-periodic modules  over commutative Noetherian local rings have been studied. We show that under certain assumptions on a pair of modules $\left(M,N \right)$ with $M$ two-periodic, the natural map $M \otimes_R N \to Hom_R(M^*,N)$ is an isomorphism. As a consequence, we have that the Auslander's depth formula holds for such a pair. Celikbas et al. recently showed the Huneke-Wiegand conjecture holds over one-dimensional domain for two-periodic modules. We generalize their result to the case of two-periodic module with rank over any one-dimensional local ring. More generally, under certain assumptions on the modules, we show that a pair of modules over an one-dimensional local ring has non-zero torsion if and only if they are Tor-independent.
	\end{abstract}
	 
	\maketitle
	
	\section{Introduction} 
	Throughout, $R$ denotes a commutative Noetherian local ring. All $R$-modules are assumed to be finitely generated. 
	
	The tensor product $M \otimes_RN$ of two modules $M$ and $N$ usually contains a non-zero torsion submodule. The assumption that this tensor product is ``nice"; for instance, torsion-free or reflexive, forces strong conditions on the modules $M$ and $N$. In this note, we aim to study various properties of tensor product of modules, where one of the modules is two-periodic.
	
	Let $M$ be a $R$-module, and ${\bf F}: \ \ \rar F_1 \rar F_0 \rar M \rar 0$ be a minimal free resolution of $M$. Recall that $M$ is said to be \textit{two-periodic} (cf. \cite{Eisen}) if there is a map of complexes $s: {\bf F} \rar {\bf F}$ of degree $-2$ such that $s: F_{i+2} \rar F_i$ is an isomorphism for $i \geq 0$.  Equivalently, $M$ is two-periodic if and only if $M \cong \Omega_R^2 M$, where for any $i \geq 0$, $\Omega^i_R M$ denotes the $i^{\text{th}}$ syzygy of $M$. More generally, we say a module $M$ is eventually two-periodic if the minimal free resolution of $M$ is periodic of period 2 after a certain stage, that is if $\Omega_R^s M $ is two-periodic for some $s \geq 0$. Two-periodic modules are eventually two-periodic, in particular.
	
	Eisenbud in \cite[Theorem 6.1]{Eisen} showed that over a hypersurface local ring, a maximal Cohen-Macaulay module (with no free summands) is two-periodic, and all modules are eventually two-periodic. He also showed that over complete intersection local rings, all modules with bounded Betti numbers are eventually two-periodic (cf. \cite[Theorem 4.1]{Eisen}). Later, Gasharov and Peeva gave certain criteria for when modules over Cohen-Macaulay local rings and arbitrary local rings are eventually two-periodic; see Theorem 1.2 and Proposition 3.8 of \cite{GP}. 
	
	Two-periodic modules have complexity 1 and are of reducible complexity. It has a simpler free resolution of infinite length; thus providing a rich class of modules beyond those with finite projective dimensions.

	As mentioned earlier, our goal is to study $M \otimes_R N$ when at least one of them  is two-periodic. Our first result in this direction is devoted to a better description of  $M \otimes_R N$. Let us denote the dual of a $R$-module $M$ by $M^*$. It is well known that there always exists a natural map 
 	\begin{equation}\label{temp_1}
 	\alpha: M^* \otimes_R N \longrightarrow Hom_R(M,N),
	\end{equation}
	given by $\alpha(f\otimes n)(m)= f(m) \cdot n$ for $f \in M^*$, $m \in M$, and $n \in N$. On the other hand, the natural map $M \longrightarrow M^{**}$ induces the map $$M \otimes_R N \longrightarrow M^{**}\otimes N.$$
	Replacing $M$ by $M^*$ in \cref{temp_1}, and composing both the maps, we get a natural map $$M \otimes N \to Hom_R(M^*,N).$$ This map in general need not be injective or surjective. However, under some special circumstances, the map will be an isomorphism. This gives us a better description of  	 $M \otimes N$.
 \begin{theorem}  \label{thm_1} Let $R$ be a local ring and $M$ be a two-periodic $R$-module with finite Gorenstein dimension. Let $N$ be any $R$-module such that the pair $(M,N)$ is Tor-independent over $R$. Then the natural map $$M \otimes_R N \to Hom_R(M^*,N)$$ is an isomorphism.
	\end{theorem}  
	Let us now turn to an application of \Cref{thm_1}. 
	The above theorem will show that the depth formula holds for such a pair $(M,N)$ of $R$-modules.	A pair $(M,N)$ of $R$-modules is said to satisfy the Auslander's depth formula (cf. \cite{A1}) if $$depth\, M + depth\, N = depth\, R + depth\, M \otimes_R N.$$ The depth formula is known to be true for the Tor-independent pair of modules $(M,N)$ (that is, $Tor_i^R(M,N) = 0$ for $i \geq 1$), where either $M$ or $N$ has finite projective dimension \cite[Theorem 1.2]{A1}. Huneke and Wiegand \cite[Proposition 2.5]{HW} showed that the depth formula holds for the Tor-independent pair $(M,N)$ of $R$-modules over complete intersection rings. The depth formula is shown to hold for Tor-independent pair of modules when one of the modules has finite complete intersection dimension by \cite[Theorem 2.5]{AY}, and independently by Iyenger \cite[Theorem 4.3]{Iyenger}. 		Over Gorenstein rings, it is not known if the depth formula for the Tor-independent pair $(M,N)$ holds or not. 
	In the setup of two-periodic modules, we show the following: 	
	\begin{theorem} \label{thm1_cor}
		Let $R$ be a local ring with positive depth and $M$ be a two-periodic $R$-module with finite Gorenstein dimension. Let $N$ be any $R$-module such that the pair $(M,N)$ is Tor-independent over $R$. Then the pair $(M,N)$ satisfies the depth formula.
	\end{theorem} 
	It is worthwhile to note that a two-periodic module has reducible complexity (for more details about reducible complexity, we refer \cite{Bergh}). \Cref{thm1_cor} is now immediately follows from \cite[Theorem 1.2]{depth_formula}. However, we reprove the result in an  elementary approach. 
	
	It is usually very hard to determine whether $M \otimes_R N$ has torsion or not for $R$-modules $M$ and $N$. Such investigation for $N=M^*$, in particular, leads to a famous open question of  	Huneke and Wiegand (cf. \cite[p. 473]{HW}):
	\begin{conj}[Huneke -- Wiegand] \label{conj_1}
		Let $R$ be a one-dimensional local ring. Let $M$ be a non-free and torsion-free module on $R$. Assume $M$ has rank. Then the torsion submodule of $M \otimes_R M^*$ is non-zero, i.e., $M \otimes_R M^*$ has (non-zero) torsion, where for any $R$-module $N$, $N^* := Hom_R(N,R)$.
	\end{conj}

	\noindent Huneke and Wiegand established the conjecture \ref{conj_1} for hypersurface domains in \cite[Theorem 3.1]{HW}. It is also known to be true over integrally closed domains (cf. \cite[Proposition 3.3]{A1}). The ideal version of the conjecture also holds for certain cases; for example, see \cite[Proposition 1.3]{celikbasideal} and \cite[Theorem 1.4]{GTTL}. Recently, Celikbas et al. solved \Cref{conj_1} for the class of two-periodic $R$-modules when $R$ is a one-dimensional domain. We point out that their proof can be modified to give the following:
	
	\begin{theorem} \label{thm_cuha} 
		Let $R$ be a one-dimensional local ring and $M$ be a non-zero two-periodic $R$-module that has rank. Then $M \otimes_R M^*$ has (non-zero) torsion. 
	\end{theorem}
	
	\noindent More generally, we show the following rigidity-type result:
	\begin{theorem} \label{thm_2} 
		Let $R$ be a one-dimensional local ring. Let $M$ be a generically free two-periodic $R$-module and $N$ be a torsionless $R$-module that has rank. Then $M \otimes_R N$ is torsion-free if and only if $M$ and $N$ are Tor-independent, that is, $Tor_i^R(M,N) = 0$ for $i \geq 1$.  
	\end{theorem}
	  
	 The Huneke-Wiegand conjecture, over Gorenstein rings, is in fact a special case of a celebrated conjecture of Auslander and Reiten \cite{AR}. Towards the end of the paper, we remark that the two-periodic modules over arbitrary local rings satisfy the Auslander-Reiten conjecture. \\

	\noindent{\bf Acknowledgements}.The first author is supported by the INSPIRE faculty fellowship (Ref No.: IFA21-MA 161) funded by the DST, Govt. of India. The second author is partially supported by a NET Senior Research Fellowship from UGC, MHRD, Govt. of India. Both authors are grateful to Amit Tripathi for suggesting the problem and for several fruitful discussions. The second author would like to thank IIT Hyderabad for its hospitality during the Conference on Commutative Algebra and Algebraic Geometry (CoCAAG 2023), where the discussion about the project was initiated. We thank Souvik Dey for carefully reading the earlier version of the manuscript and sharing his helpful comments. Thanks are due to him for pointing out that \Cref{thm1_cor} is a special case of \cite[Theorem 1.2]{depth_formula}.   
	
	\section{Preliminary} 
	Let $R$ be a local ring and $M$, $N$ be two finitely generated $R$-modules. We say that $M$ and $N$ are projectively equivalent, written as $M \approx N$, if there exist two projective $R$-modules $P$ and $Q$ such that $M \oplus P \cong N \oplus Q$. 
	
	Let $F_1 \rar F_0 \rar M \rar 0$ be a free (equivalently, projective) presentation of $M.$  The Auslander dual of $M$, denoted by $D(M)$,  is the cokernel of the induced map $F_0^* \rar  F_1^*$ (\cite[Definition 2.5]{AB}), where $M^*$ denotes $Hom_R(M,R)$. Dualizing the above sequence, we get the following exact sequence
	\begin{equation}\label{equ_aus_dual_exact} 
		0 \rightarrow M^* \rightarrow F_0^* \rar  F_1^* \rightarrow D(M) \rightarrow 0. 
	\end{equation}
	It is worthwhile to note that $D(M)$ depends on the free presentation of $M$; however, it is determined uniquely up to projective equivalence (cf. \cite[Proposition 4]{Masek}). 
	It immediately follows from \eqref{equ_aus_dual_exact} that $M^* \approx \Omega^2_R D(M)$.  
	We refer \cite{Masek} for a few properties of the Auslander dual used in this paper, all of which are originally stated in \cite{AB}.
	
	An $R$-module $M$ is said to be torsionless (reflexive, respectively) if the natural map $M \rar M^{\ast \ast}$ is injective (bijective, respectively). Moreover, a reflexive module is called totally reflexive if $Ext^i_R(M,R)=Ext^i_R(M^*,R)=0$ for all $i \geq 1$. It is well known that a module is torsionless if and only if it embeds inside a free $R$-module. A torsionless module is thus torsion-free. 
	A reflexive module is a second syzygy of some $R$-module. Reflexivity and torsionlessness of $M$ can be expressed as vanishing of certain cohomologies of $D(M)$. More precisely, the following sequence is exact
	\begin{equation}\label{eqn_1}
		0 \rar Ext^1_R\left(D(M),R\right) \rar M \rar M^{**} \rar Ext^2_R(D(M),R) \rar 0.
	\end{equation}
	Auslander and Bridger in \cite[Theorem 2.8]{AB} extended the exact sequence \eqref{eqn_1} functorially to the higher cohomologies and proved that, for any $k \geq 0$, there exist exact sequences 
	\begin{align}   
		\label{eqn_3} \scalemath{0.9}{0 \rar Ext_R^1\left(D\left(\Omega^{k}_RM\right), -\right) \rar Tor^R_{k}\left(M,-\right) \rar Hom_R\left(Ext_R^k(M,R), -\right) \rar   Ext_R^2\left(D\left(\Omega^{k}_RM\right), -\right)},\\
		\label{eqn_4} \scalemath{0.9}{Tor^R_2\left(D\left(\Omega^{k}_RM\right), -\right) \rar	Ext_R^k\left(M, R\right) \otimes (-) \rar Ext_R^k\left(M, -\right) \rar Tor^R_1\left(D\left(\Omega^{k}_RM\right), -\right) \rar 0.}
	\end{align} 
	
	Let $M$ be a torsionless $R$-module, and let $f_1,\cdots, f_n \in M^*$ be a minimal generating set of $M^*$. Define a map $f: M \rar R^n$ as $x \mapsto (f_1(x), \cdots, f_n(x))$. Then $f$ is an injective map. The universal pushforward of $M$ is the $R$-module $M_1 := coker(f)$. We note the following result
	
	\begin{lemma} \label{lemma_ext1_equals_tor1}
		Let $R$ be a local ring and $M$ be a torsionless $R$-module. Assume $M_1$ to be the universal pushforward of $M$. Then 
		\begin{enumerate}
			\item $\Omega^1_RD(M) \approx D(M_1)$. 
			\item $Ext^1_R\left(D(M), N\right) \cong Tor_1^R\left(M_1,N\right)$ for any $R$-module $N$. 
		\end{enumerate}
	\end{lemma}
	\begin{proof} 
		Consider the defining short exact sequence of the universal pushforward 
		$$0 \rar M \rar R^n \rar M_1 \rar 0.$$
		Applying \cite[Lemma 6]{Masek}, we get the following exact sequence (for some suitable choice of Auslander duals)
		$$0 \rar M_1^* \rar (R^{n})^* \rar M^* \rar D(M_1) \rar D(R^n) \rar D(M) \rar 0.$$
		Note that the map $\left(R^{n}\right)^*  \rar M^*$ is surjective by the construction of the universal pushforward. The above exact sequence gives rise to the following short exact sequence
		$$0\rar D(M_1) \rar D(R^n) \rar D(M) \rar 0,$$
		and hence $D(M_1)$ is projectively equivalent to $\Omega^1_R(D(M)$ as $D(R^n) \approx 0$.
		
		To see the second part, consider the exact sequence \eqref{eqn_4}, with $k = 1$. We obtain the following exact sequence 
		\begin{align*}
			\scalemath{0.85}{Tor_2^R\left(D\left(\Omega^1_RD(M) \right), N\right) \rar Ext^1_R\left(D(M), R\right) \otimes_R N \rar Ext^1_R\left(D(M), N\right) \rar Tor_1^R\left(D\left(\Omega^1_RD(M)\right), N\right) \rar 0.}
		\end{align*} 
		The module $M$ being torsionless, the second term in the above exact sequence vanishes. Thus, we have isomorphisms,
		$$Ext^1_R\left(D(M), N\right) \cong Tor_1^R\left(D\left(\Omega^1_RD(M)\right), N\right) \cong Tor_1^R\left(D\left(D(M_1)\right), N\right) \cong Tor_1^R\left(M_1,N\right),$$
		where the second isomorphism uses the projective equivalence stated in the first part of this lemma, and the last isomorphism follows from the fact that $D(D(M_1)) \approx M_1$ (cf. \cite[Remark 3, p. 5789]{Masek}). 
	\end{proof}

	A minimal free resolution ${\bf F}: \ \ \rar F_1 \rar F_0$ of a two-periodic $R$-module $M$ has simpler description. The resolution ${\bf F}$ is determined by $F_0$ and $F_1$ only; other free modules are given by the alteration of $F_0$ and $F_1$. Thus, all the information of ${\bf F}$ is encoded in the following two short exact sequences
	\begin{align}
		\label{seq_A} 0 \rar \Omega^1_RM \rar F_0 \rar M \rar 0, \\
		\label{seq_B} 0 \rar M \rar F_1 \rar \Omega^1_RM \rar 0, 
	\end{align}
	where the $M$ in the exact sequence \eqref{seq_B} is identified with $\Omega^2_RM$. As $M$ is contained inside the free module $F_1$, it is torsionless. Thus it is torsion-free, in particular. Also, when the Gorenstein dimension of the two-periodic module is finite, they become totally reflexive. This follows from applying depth lemma on \ref{seq_A} and \ref{seq_B}. 
	
	Before proceeding further, let us make a definition. Given a $R$ module $M$, we say $M$ is \textit{projectively two-periodic} if $M \approx \Omega^2_RM$. It is worthwhile to note that two-periodic modules are projectively two-periodic. The next result says for certain projectively two-periodic modules, the regular dual and the Auslander dual are projectively equivalent.
	
	\begin{lemma} \label{aus-dual} 
		Let $R$ be a local ring and $M$ be a projectively two-periodic and totally reflexive $R$-module. Then $M^* \approx D(M)$.
	\end{lemma}
	\begin{proof}
		Let $F_1 \rar F_0 \rar M \rar 0$ be a minimal presentation of $M$.  Since $M$ is totally reflexive, $Ext^1_R\left(M,R\right) = Ext^1_R\left(\Omega^1_RM,R\right)  = 0$. Then, dualizing the short exact sequence $0 \rar \Omega_R^1 M \rar F_0 \rar M \rar 0$, we get $Coker \left(M^* \hookrightarrow F_0\right)$ is $ \left(\Omega_R^1 M\right)^*$. Therefore, the exact sequence \eqref{equ_aus_dual_exact} splits into the following two short exact sequences 
		\begin{equation*}
			0 \rar M^* \rar F_0 \rar  \left(\Omega_R^1 M\right)^* \rar 0,
		\end{equation*} 
		\begin{equation*}
			0 \rar  \left(\Omega_R^1 M\right)^* \rar F_1 \rar D(M) \rar 0.
		\end{equation*} 	
		Dualizing both the short exact sequences again and using the fact that $\Omega_R^1 M$ is totally reflexive, we derive the following two short exact sequences 
		$$0 \rar \left(\Omega_R^1 M\right)^{**} \rar F_0 \rar M^{**} \rar 0, $$
		$$\text{ and \, \, } 0 \rar D(M)^* \rar F_1 \rar  \left(\Omega_R^1 M\right)^{**} \rar 0.$$
		Combining both sequences, we get
		\begin{equation}\label{eqn_temp1}
			0 \rar D(M)^* \rar F_1 \rar F_0 \rar M \rar 0.
		\end{equation}
		On the other hand, we have the natural exact sequence 
		\begin{equation}\label{eqn_temp2}
			0 \rar \Omega_R^2 M \rar  F_1 \rar F_0 \rar M \rar 0.
		\end{equation} 
		Schanuel's lemma is now applied to the pair of exact sequences \eqref{eqn_temp1} and \eqref{eqn_temp2} to conclude  $D(M)^* \approx \Omega_R^2 M$.  Since $M$ is projectively two-periodic, it follows that  $D(M)^* \approx M$. Therefore, $D(M)^{**} \approx M^*$ as projective equivalence is preserved under the process of dualizing. Also $M$ is totally reflexive; so is $D(M)$. Hence, $D(M)$ is reflexive, in particular. The result follows. 
	\end{proof} 
	
	We are now in a position to give a better description of the first syzygy of certain two-periodic modules, which will be useful in proving the main theorems.
	\begin{proposition} \label{lemma_univ_push}
		Let $R$ be a local ring and $M$ be a two-periodic totally reflexive $R$-module. Then $M^*$ is projectively two-periodic and $\Omega^1_R M \approx M_1,$ where $M_1$ is the universal pushforward of $M$.   
	\end{proposition}
	\begin{proof} 
		We first show that $M^*$ is projectively two-periodic. Let ${\bf F} \rar M$ be a minimal free resolution of $M$. Since $M$ is totally reflexive, $Ext^1_R\left(M,R\right) = Ext^1_R\left(\Omega^1_RM, R\right)  = 0$. Both the vanishing implies $Im\left(F_0^* \rar F_1^*\right) \cong \left(\Omega^1_RM\right)^*$, and thus
		\begin{equation}\label{eqn_temp_3}
			\Omega^1_RD(M) \approx \left(\Omega^1_R M\right)^*.
		\end{equation}
		Since $M$ is two-periodic and totally reflexive; dualizing both the short exact sequences \eqref{seq_A} and \eqref{seq_B}, the following short exact sequences can be obtained $$0 \rar M^* \rar F_0 ^*\rar \left(\Omega_R^1 M\right)^* \rar 0,$$ and $$ 0 \rar \left(\Omega_R^1 M\right)^* \rar F_1^* \rar M^* \rar 0.$$
		Both the exact sequences give rise to the following exact sequence $$0 \rar M^* \rar F_0 ^*\rar  F_1^* \rar M^* \rar 0.$$ On the other hand, there always exist free $R$ modules $G_1, G_2$ making the following exact sequence 
		$$0 \rar \Omega_R^2 M^* \rar G_1 \rar  G_0 \rar M^* \rar 0.$$
		Now applying Schanuel's lemma to the last two exact sequences, we conclude that $M^* \approx \Omega_R^2 M^*$, that is, $M^*$ is projectively two periodic.  
		
		It is worthwhile to note that $\Omega_R^1M$ is two-periodic and totally reflexive as $M$ is so. By the same argument as above, we conclude that $\left(\Omega_R^1M\right)^*$ is projectively two periodic, that is, $\left(\Omega_R^1 M\right)^* \approx \Omega_R^2 \left(\Omega_R^1 M\right)^*$. Being the dual of a totally reflexive module, $\left(\Omega_R^1 M\right)^*$ is totally reflexive as well. Now apply \Cref{aus-dual} to conclude $D\left(\left(\Omega_R^1 M\right)^* \right) \approx \left(\Omega_R^1 M\right)^{**}$.
		
		By \Cref{lemma_ext1_equals_tor1}, we already know that $D(M_1) \approx \Omega^1_R D(M)$. Therefore, we get 
		$$M_1 \approx D\left(D(M_1)\right) \approx D\left(\Omega^1_RD(M)\right) \approx D\left(\left(\Omega^1_R M\right)^*\right) \approx \left(\Omega^1_R M\right)^{**} \cong \Omega^1_R M,$$
		where the third equivalence follows from the projective equivalence in \cref{eqn_temp_3} and the last isomorphism uses the fact that $\Omega^1_R(M)$ is reflexive.
	\end{proof}
	We say an $R$-module $M$ is generically free if it is locally free on $Ass (R)$.	Using a result of Huneke and Wiegand \cite{HW}, we have a better description of the torsion submodule of $M\otimes_R N$ when $M$ is $q$-periodic and generically free. An $R$-module $M$ is said to be of period $q$ (or, $q$-periodic for short) if it has a periodic minimal free resolution of period $q$, i.e., if $M \cong \Omega^q_R M.$
	
	\begin{lemma} \label{lemma_3}  Let $R$ be a local ring and $M$ be a $q$-periodic $R$-module which is generically free. Furthermore, assume $N$ is a torsionless $R$-module. Then the torsion part of $M \otimes_R N$, $T\left(M \otimes_R N \right) \cong Tor_q^R \left( M,N \right)$.
	\end{lemma}
	\begin{proof} Since $M$ and $N$ are both torsionless, there exist free $R$-modules $F$ and $F'$ such that $M \subset F$ and $N \subset F'$. By a result of Huneke and Wiegand \cite[Lemma 1.4]{HW}, it follows that $T\left(M \otimes_R N \right) \cong Tor_2^R\left( F/M, F'/N \right)$. Tensoring the sequence $0 \rar M \rar F \rar F/M \rar 0$ with $F'/N$, we get $$Tor_2^R\left(F/M, F'/N \right) \cong Tor_1^R \left( M, F'/N \right).$$ By periodicity of $M$, $Tor_1^R\left(M, F'/N \right) \cong Tor_{q+1}^R \left(M, F'/N\right).$ Tensoring the short exact sequence $0 \rar N \rar F' \rar F'/N \rar 0$ by $M$, we see that $Tor_{q+1}^R\left(M, F'/N \right)$ is isomorphic to $Tor_q^R\left(M,N\right).$  Thus, we have shown that 
		$$T\left(M \otimes_R N\right) \cong Tor_2^R\left(F/M, F'/N\right) \cong Tor_q^R\left(M,N\right).$$
	\end{proof} 
	\section{Proof of the main results}
	In this section, we start by giving the proof of \Cref{thm_1} and \Cref{thm1_cor}. We also discuss the torsionness of $M \otimes_R N^*$ and give some interesting results in the two-periodic case. We discuss the Hochster's theta function, using which we show Theorem \ref{thm_cuha}. We end by making an observation about the Auslander-Reiten conjecture.
	
	\begin{proof}[Proof of \Cref{thm_1}] 
		
		Putting $k=0$ in exact sequence \eqref{eqn_3}, we get the following exact sequence 
		$$0 \rightarrow Ext^1_R \left(D(M),N\right) \rightarrow M \otimes_R N \rightarrow Hom_R\left(M^*,N\right) \rightarrow Ext^2_R \left(D(M),N\right).$$
		We will now show that $Ext^1_R \left(D(M),N\right)=0$ and $Ext^2_R \left(D(M),N\right)=0$. Indeed, apply \Cref{lemma_ext1_equals_tor1} and \Cref{lemma_univ_push} to conclude the following
		$$Ext^1_R \left(D(M),N\right) \cong Tor_1^R\left(M_1, N\right) \cong Tor_1^R\left(\Omega^1_R(M), N\right) = Tor_2^R\left(M, N\right)= 0.$$
		To get $Ext^2_R \left(D(M),N\right)$, we dualize the short exact sequence \eqref{seq_B}, and get the following short exact sequence
		$$0 \rar \left( \Omega_R^1 M \right)^* \rar F_1^* \rar M^* \rar 0.$$ 
		The right exactness follows from the fact that $M$ is totally reflexive, and thus, $Ext^1_R\left( \Omega_R^1M,R\right)=0$. Therefore, for a suitable choice of Auslander duals, the following sequence
		$$0 \rar D\left(\Omega_R^1M\right) \rar D(F_1) \rar D(M) \rar 0$$ 
		is exact as follows from \cite[Lemma 6]{Masek}. Note that $D(F_1) \approx 0$, and hence $Ext^i_R\left( D(F_1),N\right)=0$ for $i \geq 1$ and $R$-module $N$. Thus we get 
		$$Ext^2_R(D(M), N) \cong Ext^1_R(D(\Omega_R^1M), N) \subset Tor_1^R(M,N) = 0.$$ 
		Thus, we have shown that $$ M \otimes_R N \cong Hom_R(M^*,N).$$ 
	\end{proof}
	
	\begin{proof}[Proof of \Cref{thm1_cor}]
		Since $M$ is a two-periodic module, we have $depth R = depth M$. So, it suffices to show $depth N = depth M \otimes_R N$. From Theorem \ref{thm_1}, we have  $M \otimes_R N \cong Hom_R(M^*,N)$.
		
		When the depth of $N$ is 0, we have an injection $R/m \hookrightarrow N$. Applying $Hom_R(M^*,-)$, we see $Hom_R(M^*,N)$ contains a non-zero module of depth 0. Therefore, the depth of $Hom_R(M^*, N)$ is $0$ as well. So, the depth formula is satisfied in this case.
		
		If $depth\, N > 0$, then $depth\, M \otimes_R N = depth\, Hom_R(M^*, N) > 0$. Hence $R,M,N,M \otimes_R N$ all have positive depth, so we have a common regular element. Let $d$ be the smallest dimension of $R$ for which the theorem fails to hold. Now, we go modulo the common regular element. Two-periodicity of the new quotient module is preserved by \cite[Theorem 1.1.5]{BH}, as the minimality is intact. It follows from \cite[Theorem 8.7]{Gdim} that the finiteness of the Gorenstein dimension is not affected after going modulo the regular element. Thus, by \cite[Lemma 10]{ST}, we have a contradiction on the minimality of $d$. 
	\end{proof}
	\begin{corollary} Let $R$ be a Gorenstein ring and let $M$ be a two-periodic $R$ module. Let $N$ be a finite $R$-module such that $M$ and $N$ are Tor-independent over $R.$ Then the pair $(M,N)$ satisfies the depth formula. 
	\end{corollary} 
	\begin{proof}
		On Gorenstein rings, any finitely generated $R$-module has a finite Gorenstein dimension. The result follows from \Cref{thm_1}. \end{proof}
	\subsection{Torsionness of $M \otimes N^*$}
	If $M$ is generically free, then $M \otimes_R N^*$ being torsion-free has some interesting consequences. For instance, putting $k=0$ and $N=M$ in the exact sequence \eqref{eqn_4}, we get the following exact sequence
	\begin{align} \label{eqn_5} 
		\rar Tor_2^R\left(D(M), N\right) \rar M \otimes_R N^* \xrightarrow{\alpha} Hom_R(N,M) \rar Tor_1^R\left(D(M), N\right) \rar 0.  
	\end{align}  
	Since $M$ is generically free, the $Tor_2$ term is a torsion module. Therefore, its image vanishes, and hence the natural map $$\alpha: M \otimes_R N^* \rar Hom_R(N,M)$$ is injective. 
	\begin{remark} \label{remark_3}
		When $N=M$, the map $\alpha: M \otimes_R M^* \rar Hom_R(M,M)$ was studied by Auslander and Goldman  \cite[Theorem A.1]{AG}. They showed that the surjectivity of $\alpha$ implies $M$ is free. 
	\end{remark} 
	For two-periodic modules, we can further exploit these properties of the map $\alpha.$
	\begin{lemma} \label{lemma_tor_1} Let $R$ be a local ring and $M$ be a non-zero $R$-module that is two-periodic. Then $Tor_1^R(M, M^*) \neq 0.$   
	\end{lemma}
	\begin{proof} Assume that $Tor_1^R (M,M^*) = 0$. Since $M$ is two-periodic, let ${\bf F}: \ \ \rar F_1 \rar F_0 \rar M \rar 0$ be a minimal free resolution of $M$ satisfying exact sequences \eqref{seq_A} and \eqref{seq_B}. Consider the following diagram 
		\begin{equation*}
			\begin{gathered}
				\xymatrix@C-=1.2cm@R-=2.2cm{0 \ar[r] & \Omega_R^1M \otimes_R M^*\ar[r] \ar[d]^{\beta}& F_0 \otimes_R M^* \ar[d] \ar[r]& M \otimes_R M^*\ar[r] \ar[d]^{\alpha} & 0 \\ 0 \ar[r] & Hom_R(M,\Omega_R^1M) \ar[r] & Hom_R(M, F_0) \ar[r]& Hom_R(M,M) \ar[r]  & } \end{gathered} 
		\end{equation*} Note that the vertical arrow in the middle is an isomorphism. Thus, $\beta$ is necessarily injective, thanks to the snake lemma. Consider another diagram
		\begin{equation}
			\begin{gathered}
				\xymatrix@C-=1.2cm@R-=2.2cm{ \ar[r] & M \otimes_R M^*\ar[r] \ar[d]^{\alpha}& F_1 \otimes_R M^* \ar[d] \ar[r]& \Omega^1_RM \otimes_R M^*\ar[r] \ar[d]^{\beta} & 0 \\ 0 \ar[r] & Hom_R(M,M) \ar[r] & Hom_R(M, F_1) \ar[r]& Hom_R(M,\Omega^1_R M) \ar[r]  & }
			\end{gathered} 
		\end{equation} 
		Since $\beta$ is an injective map as shown above and the middle vertical map is an isomorphism, therefore $\alpha$ is a surjective map. It follows from \cite[Theorem A.1]{AG} that $M$ is free. 
		Thus, $M$ is necessarily zero as follows from its two-periodicity. This gives us a contradiction.
	\end{proof}
	\begin{proposition} \label{prop_ext} Let $R$ be a local ring and $M$ be a two-periodic  generically free $R$-module. Let $N$ be any $R$-module. If $Ext^1_R(N,M) = 0$, then $M \otimes_R N^*$ is torsion-free. Moreover, if $Ext^1_R(N,R)=0$, then $Ext^1_R(N,M) = 0$ if and only if $M \otimes_R N^*$ is torsion-free. 
	\end{proposition}
	\begin{proof} 
		Let ${\bf F}: \ \ \rar F_1 \rar F_0 \rar M \rar 0$ be a minimal free resolution of $M$ satisfying exact sequences \eqref{seq_A} and \eqref{seq_B}. We have the following diagram
		\begin{equation*}
			\begin{gathered}
				\xymatrix@C-=1.2cm@R-=2.2cm{\ar[r] & \Omega_R^1M \otimes_R N^*\ar[r] \ar[d]^{\delta}& F_0 \otimes_R N^* \ar[d] \ar[r]& M \otimes_R N^*\ar[r] \ar[d]^{\gamma} & 0 \\ 0 \ar[r] & Hom_R(N,\Omega_R^1M) \ar[r] & Hom_R(N, F_0) \ar[r]& Hom_R(N,M) \ar[r]  & }
			\end{gathered} 
		\end{equation*}
		Since the middle vertical map is an isomorphism, by snake lemma, $Ker(\gamma) = Coker(\delta).$ Consider the leftmost column \textemdash \,  our assumptions imply that $Hom_R(N,M)$ is torsion-free, so $$T\left(M \otimes_R N^*\right) = T\left(Ker \, \gamma\right) = T\left(Tor_2^R\left(D(N), M\right) \rar M \otimes_R N^*\right)$$ (where for any $R$-module, $N$, $T(N)$ denotes the torsion submodule of $N$). Note that $Tor_2^R(D(N), M)$ is a torsion module. Thus, $M \otimes_R N^*$ is torsion-free if and only if the map $\gamma$ is injective (equivalently, if and only if $\delta$ is surjective).
		
		Consider the following diagram
		\begin{equation} \label{dgm_2}
			\begin{gathered}
				\xymatrix@C-=1.2cm@R-=2.2cm{ \ar[r] & M \otimes_R N^*\ar[r] \ar[d]^{\gamma}& F_1 \otimes_R N^* \ar[d] \ar[r]& \Omega^1_RM \otimes_R N^*\ar[r] \ar[d]^{\delta} & 0 \\ 0 \ar[r] & Hom_R(N,M) \ar[r] & Hom_R(N, F_1) \ar[r]& Hom_R(N,\Omega^1_RM) \ar[r]  & }
			\end{gathered} 
		\end{equation} 
		Here again, the middle map is an isomorphism. 
		If $Ext^1_R(N,M)=0$, the map $Hom_R(N, F_1) \rar Hom_R (N,\Omega^1_RM)$ is onto. It follows from the commutativity of the right diagram of \eqref{dgm_2} that $\delta$ is surjective. If we assume  $Ext^1_R(N,R)=0$, then $$Coker\left( Hom_R(N, F_1) \rar Hom_R (N,\Omega^1_RM)\right) = Ext^1_R(N,M).$$
		The result follows.
	\end{proof}
	
	\begin{eg}
		Consider the one-dimensional hypersurface ring $R= k[[X,Y]]/(XY)$, where $k$ is a field. Let $M = R/(X).$ Then $M$ is a maximal Cohen-Macaulay module. Clearly, $M$ is 2-periodic and admits a minimal free resolution of the form $$\rar R \xrightarrow{Y} R \xrightarrow{X} R \rar M \rar 0.$$ One can show that $M$ is generically free,  $Hom_R(M, R) \cong M$ and $Hom_R(M,\Omega^1_R M) = 0$. In particular, $M \otimes_R M^* \cong M$ is torsion-free and $Ext^1_R(M,M) = 0$.
	\end{eg}
	\begin{remark}
		When $R$ is a Gorenstein ring, the one-dimensional version of \Cref{prop_ext} follows from a more general result of Huneke and Jorgensen \cite[Theorem 5.9]{HJ}, which explores the connection between the vanishing of certain $d$ consecutive Ext modules and the Cohen-Macaulayness of a tensor product.
	\end{remark}
	\begin{remark} \label{remark_1} With assumptions as in \Cref{prop_ext}, and putting $N = M$, we observe that if $M \otimes_R M^*$ is torsion-free, then the top exact sequence in diagram \eqref{dgm_2} shows that $Tor_1^R\left(\Omega^1_RM, M^*\right) = 0$. In particular, $Tor_2^R\left(M, M^*\right) \cong Tor_1^R\left(\Omega^1_RM, M^*\right) = 0$. We will need this observation later. 
	\end{remark}


		
	
		\subsection{Hochster's theta invariant}
	
	In \cite{CUHA}, Celikbas et al. have given a generalization of Hochster's theta invariant (cf. \cite{H81}) for certain two-periodic modules. In this section, we briefly recall this generalization and some properties. We use these methods to prove a variant of \cite[Theorem 1.3]{CUHA}.
	\begin{defn}  \label{defn_hoch}
		Let $R$ be a one-dimensional local ring and $M$ be a $R$-module. Furthermore, assume $M$ is two-periodic and generically free. It follows that for any $R$-module $N$, and for any $i \geq 1$, $Tor_{i}^R (M,N) \cong Tor_{i+2}^R (M,N)$ and $Tor_{i}^R (M,N)$ has finite length. Then the Hochster's theta invariant for the pair $(M,N)$, denoted by $\Theta^R(M,N)$, is defined as 
		$$\Theta^R (M,N) := \lambda \left(Tor_{2n}^R \left( M,N \right)\right) - \lambda\left(Tor_{2n-1}^R \left(M,N\right)\right),$$ for some $n \geq 1$. Here $\lambda(G)$ denotes the length of the module $G$. It is clear that $\Theta^R (M,N)$ is well-defined.
	\end{defn}
	
	\begin{remark} \label{theta_additive} Our definition is subsumed by the more general definition given in \cite{CUHA}. Thus, the general properties of theta invariant that they have discussed hold for our case as well. In particular, for any short exact sequence of $R$-modules,
		$ 0 \rar N_1 \rar N_2 \rar N_3 \rar 0 $, we have 
		$$ \Theta^R (M,N_2) = \Theta^R (M,N_1) + \Theta^R (M,N_3).$$ See \cite[Theorem 3.2]{CUHA} for a proof of this fact. 
	\end{remark}
	\begin{defn}
		We say a pair of modules $(M,N)$ over the ring $R$ is Tor-rigid if $Tor_i^R(M,N)=0$ for some $i \geq 1$ implies $Tor_j^R(M,N)=0$ for all $j \geq i$.
	\end{defn}
	
	We denote by $\overline{G} (R)_\mathbb{Q}$ the reduced Grothendieck group with rational coefficients, that is, $\overline{G} (R)_\mathbb{Q}=\left(G(R)/ \mathbb{Z}\cdot[R] \right) \otimes_{\mathbb{Z}} \mathbb{Q}$, where $G(R)$ is the Grothendieck group of finitely generated $R$-modules and $[R]$ denotes the class of $R$ in $G(R)$. 
	
	\begin{lemma} \label{lemma_rigid}
		Let $R$ be a one-dimensional local ring and $M$ be a non-zero two-periodic generically free $R$-module. Let $N$ be a $R$-module which has rank. Then $\Theta^R(M,N)$ vanishes. In particular, the pair $(M,N)$ is Tor-rigid. 
	\end{lemma}
	\begin{proof} Note that the theta invariant induces a map
		$$ \Theta^R (M,-) : \overline{G} (R)_\mathbb{Q} \rar \mathbb{Q}.$$ This map is in fact well-defined as follows from \cite[corollary 3.3]{CUHA}. Since $N$ is a module with rank over a one-dimensional local ring, it follows from \cite[Proposition 2.5]{CD}, as stated in \cite[2.14]{CUHA}, that $[N]=0$ in $\overline{G} (R)_\mathbb{Q}$, where $[N]$ denotes the class of $N$ in $\overline{G} (R)_\mathbb{Q}$. Hence, $\Theta^R (M,N)$ vanishes. Since $M$ is two-periodic, this implies that the pair $(M,N)$ is Tor-rigid.
	\end{proof}

	
	\begin{proof}[Proof of \Cref{thm_cuha}]
		Assume the contrary, i.e., $M\otimes_R M^*$ is torsion-free. By \Cref{remark_1}, the hypothesis implies that $Tor_2^R(M,M^*) = 0$. Since $M^*$ has rank, by \Cref{lemma_rigid}, $\Theta^R(M,M^*)$ vanishes. Therefore, $Tor_1^R(M,M^*) = 0$ which contradicts  \Cref{lemma_tor_1}.
	\end{proof}

	\begin{proof}[Proof of \Cref{thm_2}]
		By \Cref{lemma_3}, we know that $M \otimes_R N$ is torsion-free if and only if $T(M \otimes_R N) \cong Tor_2^R(M,N) = 0$. Moreover, \Cref{lemma_rigid} immediately yields that the pair $(M,N)$ is Tor-rigid. Therefore,  $Tor_2^R(M,N) = 0$ if and only if $Tor_j^R(M,N) = 0$ for $j \geq 2$. We get
		$$Tor_1^R\left(M,N\right) \cong Tor_1^R\left(\Omega_R^2 M,N\right) \cong Tor_3^R\left(M,N\right) =  0,$$
		where the first isomorphism follows from the two-periodicity of $M$. This completes the proof.
	\end{proof}
	\subsection{Auslander-Reiten Conjecture}
	
	Let us recall the famous Auslander-Reiten conjecture \cite{AR}. 
	 \begin{conj}
		Let $R$  be a local ring and $M$ be a $R$-module. If $Ext^i_R(M,M \oplus R)=0$ for all $i \geq 1$, then $M$ is free.   
	\end{conj}
	
	\noindent This conjecture is known to be true for local complete intersection rings, see \cite[Proposition 1.9]{ARC_CI}. However, it is open over arbitrary local Gorenstein rings. We observe that it holds for two-periodic modules over arbitrary local rings. 
	
	\begin{proposition}
		Let $R$ be a local ring and $M$ be an eventually two-periodic $R$-module. Let $s$ be the smallest non-negative integer such that $\Omega^s_R M=: N$ is two-periodic. If $Ext^i_R(N,N \oplus R)=0$ for all $i \geq 1$, then  $pd_R M \leq s$. 
	\end{proposition}
	
	\begin{proof} Let ${\bf F}: \ \ \rar F_1 \rar F_0 \rar N \rar 0$ be a minimal free resolution of $N$ satisfying exact sequences \eqref{seq_A} and \eqref{seq_B}. We have the following diagram
		\begin{equation*} 
			\begin{gathered}
				\xymatrix@C-=1.2cm@R-=2.2cm{\ar[r] & \Omega_R^1 N \otimes_R N^*\ar[r] \ar[d]^{\delta}& F_0 \otimes_R N^* \ar[d] \ar[r]& N \otimes_R N^*\ar[r] \ar[d]^{\gamma} & 0 \\ 0 \ar[r] & Hom_R(N,\Omega_R^1 N) \ar[r] & Hom_R(N, F_0) \ar[r]& Hom_R(N,N) \ar[r]  & }
			\end{gathered} 
		\end{equation*} 
		Clearly, if $Ext^1_R(N,\Omega^1_R N)=0$, $\gamma$ is surjective. Thus $N$ is free as follows from \cite[Theorem A.1]{AG}.
		
		Now apply $Hom_R(N,\_ )$ to the exact sequence $0 \to N \to F_1 \to \Omega^1_R N \to 0$. Since $Ext^{1}_R (N,R)=Ext^{2}_R (N,R)=0$, we get $Ext^1_R(N,\Omega^1_R N) \cong Ext^2_R(N,N)$. The result now follows from the assumption on $N$. 
	\end{proof}

\begin{remark}
If $M$ is two-periodic, $s=0$. Then $M$ is necessarily free. This asserts the Auslander-Reiten conjecture.
\end{remark}	
	
	\bibliographystyle{abbrv}
	\bibliography{ref}
\end{document}